\documentclass[11pt, reqno]{amsart}

\usepackage{amssymb,latexsym,amsmath,amsfonts}
\usepackage{mathrsfs}
\usepackage{graphicx}
\usepackage{upgreek}
\usepackage[usenames]{color}

\numberwithin{equation}{section}

\allowdisplaybreaks

\theoremstyle{definition}
\newtheorem{definition}{Definition}[section]

\theoremstyle{remark}
\newtheorem{remark}[definition]{Remark}

 \theoremstyle{plain}
\newtheorem{theorem}[definition]{Theorem}
\newtheorem{result}[definition]{Result}
\newtheorem{lemma}[definition]{Lemma}
\newtheorem{proposition}[definition]{Proposition}

\newtheorem{obs}[definition]{Observation}

\newcommand{\lam}{\lambda}

\newcommand{\zt}{\zeta}

\newcommand{\OM}{\Omega}
\newcommand{\D}{\mathbb{D}}
\newcommand{\I}{\mathbb{I}}

\newcommand{\hol}{\mathcal{O}}

\newcommand\hyper[2]{\left|\frac{{#1}-{#2}}{1-\overline{{#2}}{#1}}\right|}
\newcommand{\mobi}{\mathcal{M}_{\mathbb{D}}}
\newcommand{\Sn}{S_n(\Omega)}
\newcommand{\Z}{\mathbb{Z}}
\newcommand{\nat}{\mathbb{N}}
\newcommand{\bcdot}{\boldsymbol{\cdot}}
\newcommand\intgR[2]{[{#1}\,.\,.\,{#2}]}

\newcommand\minpo[1]{\boldsymbol{{\sf M}}_{{#1}}}

\newcommand\intf[4]{\genfrac{#1}{#2}{0.5pt}{0}{#3}{#4}}
\newcommand{\lrarw}{\longrightarrow}

\newcommand\sub[2]{\underset{#1}{#2}}

\newcommand{\C}{\mathbb{C}}

\makeatletter
\newcommand*{\rom}[1]{\expandafter\@slowromancap\romannumeral #1@}
\makeatother

\begin{document}

\title[Non-homogeneous Matricial domains]{Certain non-homogeneous matricial domains and Pick--Nevanlinna interpolation problem}

\author{Vikramjeet Singh Chandel}
\address{Department of Mathematics, Harish-Chandra Research Institute, Prayagraj (Allahabad), India}
\email{vikramjeetchandel@hri.res.in}

\keywords{symmetrized product, invariant pseudo-distances, holomorphic automorphisms and proper maps, spectral unit ball, holomorphic functional calculus, Pick--Nevanlinna interpolation problem}
\subjclass[2010]{Primary:  32H35, 30E05, 47A56; Secondary: 32F45, 47A60}

\begin{abstract}
In this article, we consider certain matricial domains that are naturally associated to
a given domain of the complex plane. A particular example of such domains is the {\em spectral unit ball}. 
We present several results for these matricial domains. Our
first result shows\,--\,generalizing a result of Ransford-White for the spectral unit ball\,--\,that
the holomorphic automorphism group of these matricial domains does not act transitively. We also consider $2$-point and $3$-point
Pick--Nevanlinna interpolation problem from the unit disc to these matricial domains.
We present results providing necessary conditions for the existence of a holomorphic {\em interpolant} for these problems.
In particular, we shall observe that these results are generalizations of the results provided by Bharali and Chandel related to these problems.

\end{abstract}
\maketitle

\section{Introduction and statement of results}\label{S:intro}
Let $\D$ denote the open unit disc in the complex plane $\C$ centered at $0$. Given a domain
$\Omega\subseteq\C$ and $n\in\mathbb{N}, n\geq 2$, we define:
\[
 \Sn:=\{A\in M_n(\C): \sigma(A)\subset\Omega\}.
\]
Here, $M_n(\C)$ denotes the set of all $n\times n$ complex matrices, and $\sigma(A)$ of a matrix $A\in M_n(\C)$ denotes the set of eigenvalues of $A$.
It is not difficult to check that $S_n(\Omega)$ is an open, connected subset of $M_n(\C)$. Hence,
by definition, considered as a subset of $\C^{n^2}$, it is a domain.
In the case $\Omega=\D$, the set $S_n(\Omega)$ is known in the literature as the spectral unit ball and is
denoted by $\Omega_n$.
\smallskip

In the first part of this article, we shall establish that the holomorphic automorphism group
of $S_n(\Omega)$, denoted by $Aut(\Sn)$, does not act transitively on $\Sn$ for ceratain
planar domains $\Omega$ including all bounded domains.
In particular, this result generalizes an analogous result due to 
Ransford-White (see \cite{tjr&mcw:hsmsub91}) about the spectral unit ball $\Omega_n,\,n\geq 2$.
To state this result precisely, we need to introduce a few more objects associated naturally to $S_n(\Omega)$.
\smallskip

First, we consider the {\em symmetrization map} $\pi_n:\C^n\lrarw\C^n$ defined by:
\[
\pi_n(z):=\big(\pi_{n,\,1}(z),\dots,\pi_{n,\,j}(z),\dots,\pi_{n,\,n}(z)\big),
\]
where $\pi_{n,\,j}(z)$ is the $j$-th elementary symmetric polynomial. In other words, 
if we write $\pi_{n,\,j}(z_1,\dots,z_n)=\mathscr{S}_j(z_1,\dots,z_n)$, then $\mathscr{S}_j$ satisfy:
\[
\prod_{j=1}^n(t-z_j)=t^n+\sum_{j=1}^n (-1)^j\mathscr{S}_j(z_1,\dots,z_n)\,t^{n-j}, \ \ \ t\in\C.
\]
The symmetrization map $\pi_n$ is a proper holomorphic map from $\C^n$ to $\C^n$.
Consider the {\em $n$-th symmetrized product of $\Omega$}, defined by:
\[
\Sigma^{n}(\Omega):=\pi_n(\Omega^n).
\]
It is easy to see that $\Sigma^{n}(\Omega)$ is a domain in $\C^n$.
\smallskip

We need to introduce a few more objects before we state our first result. 
Given a holomorphic self-map $\phi$ of $\Omega$, there is a holomorphic self-map that $\phi$ induces on $\Sigma^n(\Omega)$; namely: 
$\Sigma^n\phi : \Sigma^n(\Omega)\lrarw\Sigma^n(\Omega)$ defined by 
\[
 \Sigma^n\phi\,\big(\pi_n(z_1,\dots,
  z_n)\big):=\pi_n(\phi(z_1),\dots,\phi(z_n)) \ \ \ \forall \  (z_1,\dots,z_n)\in\Omega^n,
\]
where $\pi_n$ is the symmetrization map as before. Given $n\in\nat,\, n\geq 2$ and a family $\mathscr{A}_n\subset\hol(\Sigma^n(\OM),\,\Sigma^n(\OM))$ of holomorphic self-maps of $\Sigma^n(\OM)$, consider the following:
\begin{itemize}
 \item[$(\sf{P})$] for every $\Phi\in\mathscr{A}_n$, there exists a $\phi\in\hol(\OM,\,\OM)$ such that $\Phi\equiv\Sigma^n\phi$.
\end{itemize}
Observe that if $\mathscr{B}_n\subset\mathscr{A}_n\subset\hol(\Sigma^n(\OM),\,\Sigma^n(\OM))$ then it is obvious that if $\mathscr{A}_n$ satisfies the
property $(\sf{P})$ above then so does $\mathscr{B}_n$. We are now in a position to state our first main theorem as alluded to in the second paragraph above:

\begin{theorem}\label{T:holautmatdom}
Let $\Omega\subset\C$ be a domain and let $n\in\nat$, $n\geq 2$, be such that we have $\#(\C\setminus\Omega)\geq~2n$. Suppose the
holomorphic automorphism group of the $n$-th symmetrized product of 
$\Omega$, $Aut(\Sigma^n(\OM))$, has the property $(\sf{P})$ above. Then for every $\Psi\in Aut(\Sn)$ there exists a holomorphic automorphism
$\psi$ of $\Omega$ such that 
\[
 \sigma(\Psi(A))=\psi(\sigma(A)) \ \ \ \forall A\in\Sn.
\]
In particular, $Aut(\Sn)$ does not act transitively on $\Sn$.
\end{theorem}

\noindent We shall present a proof of the above theorem in Section~\ref{S:plms&proof}. We first record the following important remark regarding Theorem~\ref{T:holautmatdom}.

\begin{remark}\label{Rm:exholautmatdom}
Let $\Omega\subset\C$ be a bounded domain, then the cardinality condition in the above theorem is trivially satisfied.
It is a nontrivial result of Chakrabarty--Gorai \cite[Corollary~1.3]{dc&sg:ft&hmsppd15}\,---\,who
generalized the analogous result of Edigarian--Zwonek \cite[Theorem~1]{ae&wz:gsp05}
for $\Sigma^n(\D)$\,---\,that the family of {\bf proper} holomorphic self-maps of $\Sigma^n(\Omega)$ satisfies the property $(\sf{P})$ above.
Since the later family contains $Aut(\Sigma^n(\Omega))$, we see that for every bounded domain $\Omega$, the conclusion of the above theorem holds true.
\smallskip

The essence of the cardinality condition that features in the statement of the above theorem is that it is the condition that guarantees
the Kobayashi hyperbolicity of the domains $\Sigma^n(\Omega)$ (see Section~\ref{S:plms&proof} for the definition of
Kobayashi hyperbolicity). Infact, as it turns out not only the cardinalty condition implies the Kobayashi 
hyperbolicity of these domains but they also become Kobayashi complete; see  \cite[Theorem~16]{wz:ftpspcm20}. At this point, we do not know if the 
cardinality condition itself implies that $Aut(\Sigma^n(\Omega))$ satisfies the condition $({\sf P})$ above.
\end{remark}

One could also ask what happens if the map $\Psi$ in the statement of Theorem~\ref{T:holautmatdom}
is a proper holomorphic map instead of a holomorphic automorphism? We address this question too here,
and we show that under certain restrictions on the domain $\Omega$; namely: $\Omega$ is a hyperconvex domain for 
which the family of proper holomorphic self-maps of $\Sigma^n(\Omega)$ satisfies the condition $(\sf{P})$ above then an analogous
result similar to the conclusion of Theorem~\ref{T:holautmatdom} holds true. 
\smallskip

Before we state this result, we recall that a domain $D\subset\C^n$ is called {\em hyperconvex} if there exists a 
negative plurisubharmonic exhaustion function on $D$.
We shall see in Section~\ref{S:plms&proof} that if a domain $\Omega\subset\C$ is hyperconvex then the cardinality of $\C\setminus\Omega$ cannot be
finite. In particular, hyperconvex domains satisfy the condition on the cardinality of  $\C\setminus\Omega$ as in the statement of Theorem~\ref{T:holautmatdom}.
\smallskip

Now we state our second result related to the proper holomorphic self-maps of $\Sn$.

\begin{theorem}\label{T:holpropmatdom}
Let $\Omega\subset\C$ be a hyperconvex domain and let $n\in\nat$, $n\geq 2$, be given.
Suppose the family of proper holomorphic self-maps of $n$-th symmetrized product of 
$\Omega$, $\Sigma^n(\Omega)$, satisfies the condition $\sf(P)$ above. 
Then for every proper holomorphic map $\Psi : \Sn\lrarw\Sn$, there exists a
proper holomorphic self-map $\psi$ of $\Omega$ such that 
\begin{equation}\label{E:specprop}
\sigma(\Psi(A))=\psi(\sigma(A)) \ \ \ \forall A\in\Sn.
\end{equation}
\end{theorem}
\noindent{We shall present a proof of Theorem~\ref{T:holpropmatdom} in Section~\ref{S:plms&proof}. Our 
proof is motivated from the proof of Theorem~17 in \cite{ae&wz:gsp05}. We state here some important observations relevant to Theorem~\ref{T:holpropmatdom}.}

\begin{remark}\label{Rm:prophol=bihol}
Observe due to the result by Chakrabarty--Gorai (see Remark~\ref{Rm:exholautmatdom}), the above theorem applies
to any bounded hyperconvex domain. The punctured disc $\D^{*}:=\{\zt\in\D\,:\,\zt\neq 0\}$ is an example of a bounded domain that is not hyperconvex. 
So, although the family of proper holomorphic self-maps of $\Sigma^n(\D^{*})$ satisfies the property $(\sf{P})$,
we {\bf cannot} infer\,---\,appealing to the above theorem\,---\,that every proper holomorphic self-map $\Psi$ of $S_n(\D^{*})$ satisfies 
\eqref{E:specprop} for some proper holomorphic self-map $\psi$ of $\D^{*}$. Indeed, it would be interesting to find a counterexample in this case.
We also do not know whether $\Omega$ being hyperconvex itself implies that the property $(\sf{P})$ is satisfied for the family of proper holomorphic
self-maps of $\Sigma^n(\Omega)$.
\smallskip

It is also interesting at this point to recall an Alexander-type of result due to Zwonek for the spectral unit ball $\OM_n$. Alexander in \cite{Alex:PHMincomplex77} proved that all proper holomorphic self-maps of Euclidean unit ball $\mathbb{B}^n,\,n\geq 2$, are biholomorphisms. Zwonek in \cite{wz:phmspecunitball} proved that all proper holomorphic self-maps of  $\OM_n\,n\geq 2$ are biholomorphisms. The proof as given by Zwonek of this 
result crucially uses the relation \eqref{E:specprop} and the properties of proper holomorphic maps between analytic varieties.
It will be of interest to find out if an Alexander-type result holds true for $\Sn$ where $\Omega$ is a bounded hyperconvex domain. 
\end{remark}

\begin{remark}\label{Rem:finconcom}
Let $\Omega\subset\C$ be a domain such that $\C\setminus\Omega$ has finitely many connected components none of which is a single point.
It follows that for a fixed point $p\in\OM$, the function
$-g_{\OM}(\bcdot\,;\,p)$, where $g_{\OM}(\bcdot\,;\,\bcdot)$ denotes the Green's function for $\OM$,
is a negative subharmonic exhaustion function on $\OM$ whence $\OM$ is hyperconvex (see \cite[Chapter~4]{Ransford:pottheocomp}).
Futhermore, by Koebe's uniformization theorem for finitely connected domains,
these domains are biholomorphic to a domain all of whose boundary components are circle. 
Hence, by \cite[Corollary~1.6]{gbibddjj:phmsprs18}, the proper holomorphic self-maps of $\Sigma^n(\OM)$
satisfies the property $\sf(P)$ above. 
\end{remark}

\begin{remark}
Let $\Omega$ be a domain as in Remark~\ref{Rem:finconcom} and let $p\in\nat$ denote the number of connected components of 
$\C\setminus\Omega$. Mueller-Rudin showed in \cite{Mueller&Rudin:phsmpr91} 
that when $p\geq 3$
then the set of all proper holomorphic self-maps of $\Omega$ coincides with the set of all holomorphic 
automorphism of $\Omega$, and that the later set is a finite set. In the case when $p=2$, 
it is a fact that $\Omega$ is biholomorphic to an annulus $A_r:=\{z\in\C:r<|z|<1\}$ for some $r, 0<r<1$.
In the case of annulus it is also known (see \cite{Mueller&Rudin:phsmpr91} for a reference) that every proper holomorphic self-map
is a holomorphic automorphism. In \cite{Mueller&Rudin:phsmpr91}, the authors also constructed interesting domains of infinite connectivity that possess non-trivial proper holomorphic
self-mappings. The problem about the form of proper holomorphic self-mappings of their symmetrized product is interesting in its own.
\end{remark}

The second part of this article is devoted to the following Pick-Nevanlinna interpolation type problem.
\begin{itemize}
\item{} Given $\{(\zt_j,\,W_j)\in\D\times\Sn: 1\leq j\leq N\}$, $N\geq 2$ and $\zt_j$'s being distinct, find
necessary and sufficient conditions for the existence of a holomorphic map $F\in\hol(\D,\,\Sn)$ such that
$F(\zt_j)=W_j$ for all $j=1,\dots,N$.
\end{itemize}
\noindent{In the case when such a function $F$ exists, we shall say that $F$ is an {\em interpolant} of the
data $\{(\zt_j,\,W_j)\in\D\times\Sn: 1\leq j\leq N\}$.}
\smallskip

In this article, we shall only consider the above problem when $N=2$ or $N=3$. Starting with $N=2$,
we shall provide a necessary condition for the existence of a holomorphic interpolant. This necessary
condition will remind the reader of the classical Schwarz lemma in one complex variables. But before we state this result, we need to introduce the following pseudo-distance.
\smallskip

The {\em Carath\'{e}odory pseudo-distance}, denoted by $C_{\OM}$, on a domain $\OM$ in $\C$ is defined by:
\vspace{0.1cm}
\begin{equation}\label{E:defn_carathdist}
C_{\Omega}(p,\,q):=\sup\{\mathcal{M}_{\D}(f(p),\,f(q))\,:\,f\in\hol(\Omega,\,\D)\}.\vspace{0.1cm}
\end{equation}
Here and elsewhere in this article 
$\mobi(z_1,z_2)$ is 
the {\em M{\"o}bius distance} between $z_1$ and $z_2$, defined as:
\[
 \mobi(z_1,z_2) \ := \ \hyper{z_1}{z_2} \quad\forall z_1,z_2\in\D.\vspace{0.1cm}
\]
The reader will notice that we have defined $C_{\OM}$ in terms of the M{\"o}bius distance rather than the hyperbolic
distance on $\D$. This is done purposely because {\em most} conclusions in metric geometry that rely on $C_{\OM}$ are essentially unchanged 
if $\mathcal{M}_{\D}$ is replaced by the hyperbolic distance on $\D$ in \eqref{E:defn_carathdist}, and because the  M{\"o}bius distance
arises naturally in the proofs of our theorems. 
\smallskip

A domain $\Omega\subset\C$ will be called {\em Carath\'{e}odory hyperbolic} if $C_{\OM}$ is a distance
in the sense of metric spaces. It is easy to see that $\Omega$ is Carath\'{e}odory hyperbolic if and only
if $H^{\infty}(\Omega)$, the set of all bounded holomorphic functions in $\Omega$, separates points in
$\Omega$; e.g. every bounded domain is Carath\'{e}odory hyperbolic. In what follows, we shall always 
consider domains $\Omega$ that are Carath\'{e}odory hyperbolic.
\smallskip

We now present our first result concerning the interpolation
problem above when $N=2$.
\begin{theorem}\label{T:twopointT}
Let $F\in\hol(\D,\,S_n(\Omega))$, $n\geq 2$, and let $\zt_1,\zt_2\in\D$. 
Write $W_j=F(\zt_j)$, and if $\lam\in\sigma(W_j)$, then let $m(\lam)$ denote the multiplicity of $\lam$ as a
zero of the minimal polynomial of $W_j$. Then:
\begin{equation}\label{E:2SchwarzIneq}
\max\left\{\max_{\mu\in\sigma(W_2)}\prod_{\lam\in\sigma(W_1)}C_{\Omega}(\mu,\lam)^{m(\lam)}, 
\ \max_{\lambda\in\sigma(W_1)}\prod_{\mu\in\sigma(W_2)}C_{\Omega}(\lam,\mu)^{m(\mu)}\right\} \ 
\leq \ \hyper{\zt_1}{\zt_2}.\vspace{0.1cm}
\end{equation}
\end{theorem}
\noindent{
We shall present our proof of Theorem~\ref{T:twopointT} in Section~\ref{S:twopointT}. When $\Omega=\D$, we know that $C_{\Omega}(z_1,z_2)= \mobi(z_1,z_2)$. Substituting this into 
the inequality \eqref{E:2SchwarzIneq} establishes
Theorem~1.5 in \cite{gb:itplSpUb07} which is an 
important result related to the $2$-point interpolation problem from $\D$ to $\Omega_n$. Hence
Theorem~\ref{T:twopointT} gives a generalization of Theorem~1.5 in \cite{gb:itplSpUb07} for 
all matricial domains $\Sn$.}
\smallskip

It is shown in \cite{gb:itplSpUb07} that when $\Omega=\D$ the above result gives a necessary 
condition for the existence  of an interpolant for the $2$-point interpolation problem that is inequivalent 
to the necessary condition that are known in the literature \cite{costara:osNPp05, ogle:thesis99}. Moreover,
in this case ($\Omega=\D$) when $n\geq3$, there exists a $2$-point data set for which \eqref{E:2SchwarzIneq} implies that the data cannot admit an interpolant whereas the condition in \cite{costara:osNPp05, ogle:thesis99} are inconclusive.
\smallskip

At this point we wish to discuss an important tool that plays a crucial role in establishing the inequality
\eqref{E:2SchwarzIneq} and is at the heart of our next theorem related to the $3$-point interpolation 
problem. We begin with the extremal problem associated to the Carath\'{e}odory
pseudo-distance $C_{\Omega}$ on a domain $\Omega$ in $\C$. Recall:
\begin{align}
C_{\Omega}(p,\,q)&:=\sup\{\mathcal{M}_{\D}(f(p),\,f(q))\,:\,f\in\hol(\Omega,\,\D)\}\nonumber\\
&=\sup\{|\,f(q)\,|\,:\,f\in\hol(\Omega,\,\D)\,:\,f(p)=0\}.\label{E:alt_def_cara}
\end{align}
The equality in \eqref{E:alt_def_cara} is due to the fact that the automorphism group of $\D$
acts transitively on $\D$ and the M{\"o}bius distance is invariant under its action.
Applying Montel's Theorem, it is easy to see that there exists a function 
$g\in\hol(\Omega,\,\D)$ such that $g(p)=0$ and $g(q)=C_{\Omega}(p,\,q)$. Such a function is
called an \emph{extremal solution} for the extremal problem determined by \eqref{E:alt_def_cara}.
It is a fact that for Carath\'{e}odory hyperbolic domains there is a unique extremal
solution (see the last two paragraphs in \cite{fsh:shwaLemInnfun69}).
Let us denote by $G_{\Omega}(p,\,q;\, \bcdot{})$
the unique extremal solution determined by the extremal problem \eqref{E:alt_def_cara}.

\begin{definition}\label{Def:genminblashprod}
Given $A\in\Sn$ and $z\in\Omega\setminus{\sigma(A)}$, consider the function:
\begin{equation}\label{E:genminblashprod}
B(A,\,z;\,\bcdot):=\prod_{\lam\in\sigma(A)}\,G_{\Omega}(\lam,\,z;\,\bcdot)^{m(\lam)}
\end{equation}
where $G_{\Omega}(\lam,\,z;\,\bcdot)$ is the unique extremal solution corresponding to the pair $(\lam,z)$
as discussed above and $m(\lam)$ is the multiplicity of $\lam\in\sigma(A)$ as a zero of the minimal polynomial of $A$.
\end{definition}
\noindent{Observe that for every $z\in\Omega\setminus{\sigma(A)}$, $B(A,\,z;\,\bcdot)\in\hol(\Omega,\,\D)$
and has a zero at each $\lam\in\sigma(A)$ of multiplicity at least $m(\lam)$.
For each $z\in\Omega\setminus{\sigma(A)}$, $B(A,\,z;\,\bcdot)$ induces, via the holomorphic functional calculus (which we will discuss in
Section~\ref{S:holo_fc}), a holomorphic map from $\Sn$ to $\OM_n$ that maps $A$ to $0\in M_n(\C)$.
This simple trick turns out to be quite important in addressing the $3$-point interpolation problem. 
\smallskip

Now we are ready to present the result of this article related to the $3$-point interpolation problem
from $\D$ to $\Sn$. In what follows, $B_{j,\,z}$ will denote the 
function as defined in \eqref{E:genminblashprod}\,---\,as well as its extension to $\Sn$\,---\,associated to the matrix $W_j$,
$j=1,2,3$.

\begin{theorem}\label{T:3pt_nec}
Let $\zt_1,\zt_2,\zt_3\in\D$ be distinct points and let $W_1,W_2,W_3\in\Sn$, $n\geq 2$. 
Let $m(j,\,\lambda)$ denote the multiplicity of $\lambda$ as a zero of the minimal polynomial
of $W_j$, $j\in\{1,2,3\}$.
Given $j,k\in\{1,2,3\}$ such that $j\not=k$, $z\in\Omega\setminus\sigma(W_k)$ and $\nu\in \D$, we write:
\[
q_z(\nu,j,k):=\max\left\{\intf{[}{]}{m(j,\,\lambda)-1}
{\mathsf{ord}_{\lambda}{B'_{k,\,z}}+1}+1:\,\lambda\in \sigma(W_j)\cap B_{k,\,z}^{-1}\{\nu\} \right\}.
\]
Finally, for each $k\in\{1,2,3\}$ let
\[
G(k):=\max\,(\{1,2,3\}\setminus\{k\}),\,\,\text{and}\,\,\,
L(k):=\min\,(\{1,2,3\}\setminus\{k\}).
\]
If there exists a map $F\in\hol(\D,\,\Sn)$ such that $F(\zt_j)\,=\,W_j$, $j\in\{1,2,3\}$,
then for each $k\in\{1,2,3\}$, and $z\in\Omega\setminus\sigma(W_k)$ we have: 
\begin{itemize}
\item either $\sigma\left(B_{k,\,z}(W_{G(k)})\right)\subset
D\left(0,\,|\,\psi_{k}(\zt_{G(k)})\,|\right)$,
$\sigma\left(B_{k,\,z}(W_{L(k)})\right)\subset D\left(0,\,|\,\psi_{k}(\zt_{L(k)})\,|\right)$ and
\begin{align*}
{}&\max\left\{\sub{\mu\in\sigma(B_{k,\,z}(W_{L(k)}))}{\max}
\prod_{\nu\in\sigma(B_{k,\,z}(W_{G(k)}))}
{\mathcal{M}_{\D}\left(\intf{}{}{\mu}{\psi_k(\zt_{L(k)})},\,\intf{}{}{\nu}{\psi_k(\zt_{G(k)})}
\right)}^{q_z(\nu,\,G(k),\,k)},\right.\\
&\left.\sub{\mu\in\sigma(B_{k,\,z}(W_{G(k)}))}{\max}\prod_{\nu\in\sigma(B_{k,\,z}(W_{L(k)}))}
{\mathcal{M}_{\D}\left(\intf{}{}{\mu}{\psi_k(\zt_{G(k)})},\,\intf{}{}{\nu}{\psi_k(\zt_{L(k)})}
\right)}^{q_z(\nu,\,L(k),\,k)}\right\}\leq\mathcal{M}_{\D}\left(\zt_{L(k)},\,\zt_{G(k)}\right)
\end{align*}

\item or there exists a $\theta_z\in\mathbb{R}$ such that
\[
{B^{-1}_{k,\,z}}\{e^{i\theta_z}\psi_{k}(\zt_{G(k)})\}\subseteq\sigma(W_{G(k)}) \  \text{and} \
{B^{-1}_{k,\,z}}\{e^{i\theta_z}\psi_{k}(\zt_{L(k)})\}\subseteq\sigma(W_{L(k)}).
\]
\end{itemize}
\end{theorem}
\noindent Here, $\psi_j$ denotes
the automorphism $\psi_j(\zt):=(\zt-\zt_j)/(1-\overline{\zt}_j\zt)^{-1}$, $\zt\in\D$, of $\D$ and
$[\bcdot]$ denotes the greatest-integer function. Given $a\in \C$ and a function $g$ that is
holomorphic in a neighbourhood of $a$, $\mathsf{ord}_a{g}$ will denote the order of vanishing of
$g$ at $a$ (with the understanding that $\mathsf{ord}_a{g} = 0$ if $g$ does not vanish at $a$).
\smallskip

\begin{remark}\label{Rm:coromaintD}
We shall present our proof of Theorem~\ref{T:3pt_nec} in Section~\ref{S:3pint}. The proof is strongly motivated from the proof of Theorem~1.4 in \cite{chandel:3spintp20}. In fact, Theorem~\ref{T:3pt_nec} above is a generalization of Theorem~1.4 in \cite{chandel:3spintp20}. This is shown in Observation~\ref{Obs:coromaintD} by explicitly computing the function
$B_{k,\,z}(\bcdot)$. It turns out that when $\Omega=\D$ the statement of Theorem~\ref{T:3pt_nec}
coincides with the statement of Theorem~1.4 in  \cite{chandel:3spintp20}. We also refer the reader to
Remark~1.5 in  \cite{chandel:3spintp20} for the discussion of how Theorem~1.4 in
\cite{chandel:3spintp20} is different from the other results that are present in the literature related
to the $3$-point interpolation problem.
\end{remark}

\section{Preliminaries and proofs of Theorem~\ref{T:holautmatdom} and 
Theorem~\ref{T:holpropmatdom}}\label{S:plms&proof}

In this section, we shall present our proofs of Theorem~\ref{T:holautmatdom} and
Theorem~\ref{T:holpropmatdom}. 
At the heart of our proofs is a result that itself is quite interesting. It is Lemma~\ref{L:prespectra} below.
But before we prove this lemma, we shall take a digression and recall the definition of Kobayashi pseudo-distance and Kobayashi hyperbolicity.
\smallskip

\subsection{The Kobayashi pseudo-distance}
Let ${D}\subset\C^n$ be a domain and let ${\sf h}$ denote the hyperbolic distance on $\D$ induced by the Poincar\'{e} metric on $\D$.
The Kobayashi pseudo-distance $K_{{D}}:D\times D\lrarw [0,\,\infty)$ is defined by:
given two points $p,q\in{D}$,
\[
K_{{D}}(p, q):=\inf\Big\{\sum_{i=1}^n {\sf h}(\zt_{i-1},\zt_i)\,:\,(\phi_1,\dots,\phi_n;\zt_0,\dots,\zt_n)\in
\mathfrak{A}(p, q)\Big\}
\]
where $\mathfrak{A}(p, q)$ is the set of all analytic chains in ${D}$ joining $p$ to $q$. Here,
$(\phi_1,\dots,\phi_n;\zt_0,\dots,\zt_n)$ is an {\em analytic chain} in ${D}$ joining $p$ to $q$ if 
$\phi_i\in\hol(\D,\,D)$ for each $i$ such that
\[
p=\phi_1(\zt_0), \ \ \phi_n(\zt_n)=q \ \ \text{and} \ \ \phi_i(\zt_i)=\phi_{i+1}(\zt_i)
\]
for $i=1,\dots,n-1$.
\smallskip

It follows from the definition that $K_{D}$ is a pseudo-distance. Using the Schwarz lemma in
one complex variable one could see that
$K_{\D}\equiv{\sf h}$. One of the most important properties that the Kobayashi pseudo-distance enjoys is the following: if $F: {D}_1\lrarw {D}_2$ is a holomorphic map, then
$K_{{D}_2}\big(F(p), F(q)\big)\leq K_{{D}_1}(p, q)$
for all $p,q\in {D}_1$.
\smallskip

A domain ${D}\subset\C^n$ is called Kobayashi hyperbolic if the pseudo-distance $K_{{D}}$ is 
a true distance, i.e., $K_{{D}}(p, q)=0$ if and only if $p=q$. The collection of all bounded domains is an
example of Kobayashi hyperbolic domains.
We refer the interested reader to \cite[Chapter~3]{JP13:jarnicki2013invariant} for a comprehensive 
account on Kobayashi pseudo-distance.
It is a fact that $K_{\C^d}\equiv 0$ for all $d\geq 1$. This is not difficult to prove but we skip the proof of 
this fact here (see \cite[Chapter~3]{JP13:jarnicki2013invariant}). The following is a generalization of 
Liouville's Theorem and is obvious:
\begin{result}\label{Res:LT}
Let $D\subset\C^n$ be a Kobayashi hyperbolic domain. Let $F:\C^d\lrarw{D}$ be a holomorphic map then $F$ is a constant function.
\end{result}

We need one more tool to state the lemma alluded to at the beginning of this section. Given a matrix $A\in M_n(\C)$, we write its characteristic polynomial as $\chi(A)(t)=t^n+\sum_{k=1}^n(-1)^k\chi_k(A)\,t^{n-k}$, where $\chi_k(A)$ are polynomials in the entries of $A$. It is obvious that this naturally leads to a map $\chi : \Sn\lrarw\Sigma^n(\Omega)$ defined by:
\[
\chi(A)=\big(\chi_1(A),\dots,\chi_n(A)\big),
\]
where $\chi_k(A)$'s are as above. Notice that $\chi$ is a holomorphic map from $\Sn$ to $\Sigma^n(\Omega)$.
\smallskip

Now we can state the lemma which implies that a holomorphic self-map of $\Sn$ preserves 
the spectra of matrices. More precisely:
\begin{lemma}\label{L:prespectra}
Let $\Omega\subset\C$ be a domain and let $n\in\nat$, $n\geq 2$, be such that $\#(\C\setminus\Omega)\geq 2n$.
Let $F:\Sn\lrarw\Sn$ be a holomorphic self-map. Then for every 
$A,\,B\in\Sn$ such that $\chi(A)=\chi(B)$, we have $\chi(F(A))=\chi(F(B))$.
\end{lemma}
\begin{proof}
Let $D$ be a diagonal matrix such that $\chi(D)=\chi(A)$. We know that there exists $C\in M_n(\C)$ and an 
strictly upper triangualr matrix $U$ such that
\[
A=\exp(-C)\,(D+U)\,\exp(C).
\]
Now consider the map $f:\C\lrarw M_n(\C)$ defined by 
\[
f(\zt):=\exp(-C\,\zt)\,(D+\zt\,U)\,\exp(C\zt) \ \ \ \forall\zt\in
\C.
\]
Notice that $\chi(f(\zt))=\chi(D+\zt\,U)=\chi(D)$, hence $f(\C)\subset\Sn$. This lets us to define the map
$\Psi(\zt):=\chi\circ F\circ f(\zt)$ for all $\zt\in\C$. It is obvious that $\Psi$ is a holomorphic map from 
$\C$ to $\Sigma^n(\Omega)$.
\smallskip

Under the condition on $\Omega$ as in the statement of the lemma, a result of Zwonek
\cite[Theorem~16]{wz:ftpspcm20} implies that $\Sigma^n(\Omega)$ is 
Kobayashi hyperbolic. Then Result~\ref{Res:LT} implies that $\Psi$ is a constant function.
Hence, $\Psi(0)=\chi(F(D))=\Psi(1)=\chi(F(A))$.
Proceeding similarly we get $\chi(F(D))=\chi(F(B))$. This establishes the lemma.
\end{proof}

The proof of the above lemma is motivated from that of Theorem~1 in \cite{tjr&mcw:hsmsub91} 
by Ransford--White. But it is a far reaching generalization of Theorem~1 in \cite{tjr&mcw:hsmsub91};
e.g., when $\Omega$ is any Carath\'{e}odory hyperbolic domain, then $\#(\C\setminus\Omega)$ cannot be
finite. In particular, Carath\'{e}odory hyperbolic domains satisfy the condition as in the statement of
Lemma~\ref{L:prespectra}. 
\smallskip

We shall now present our proof of Theorem~\ref{T:holautmatdom}.

\subsection{The proof of Theorem~\ref{T:holautmatdom}}\label{SS:pot}
\begin{proof} 
Consider a relation $G$ from $\Sigma^n(\Omega)$ into $\Sigma^n(\Omega)$ defined by:
\begin{equation}\label{E:autsym}
G(X):=\chi\circ \Psi\circ{\chi}^{-1}(\{X\}) \ \ \ \forall X\in \Sigma^n(\Omega).
\end{equation}
Here $\Psi$ is as in the statement of Theorem~\ref{T:holautmatdom}.
From Lemma~\ref{L:prespectra}, it follows that for each $X\in \Sigma^n(\Omega)$, $G(X)$ is a singleton set.
Hence $G:\Sigma^n(\Omega)\lrarw\Sigma^n(\Omega)$ is a well defined map.
\smallskip

\noindent{\bf Claim.} $G$ is holomorphic.
\smallskip

\noindent To see this, choose an arbitrary $X\in\Sigma^n(\Omega)$ and fix it. Now consider the polynomial
$P_X[\,t\,]:=t^n+\sum_{j=1}^n(-1)^jX_j\,t^{n-j}$. We define a map $\tau:\Sigma^n(\Omega)\lrarw M_n(\C)$ by setting:
\begin{equation}\label{E:tau}
\tau(X):=\mathsf{C}\big(P_X\big)
\end{equation}
where $\mathsf{C}\big(P_X\big)$ denotes the companion matrix of the polynomial $P_X$. Recall:
given a monic polynomial of degree $k$ of the form $p[\,t\,]=t^k+\sum_{j=1}^ka_j\,t^{k-j}$, where $a_j\in\C$,
the \emph{companion matrix} of $p$ is the matrix $\mathsf{C}(p)\in M_{k}(\C)$ given by
\[
\mathsf{C}(p):=
\begin{bmatrix}
\ 0  & {} & {} & -a_k \ \\
\ 1  & 0  & {} & -a_{k-1} \ \\
\ {} & \ddots  & \ddots & \vdots \ \\
\ \text{\LARGE{0}} &   & 1 & -a_{1} \
\end{bmatrix}_{k\times k}.
\]
It is a fact that $\chi(\mathsf{C}(p))(t)=p(t)$. From this, it follows that 
$\tau$ is holomorphic and $\chi\circ\tau=\mathbb{I}$ on
$\Sigma^n(\Omega)$.
This, in particular, implies that $\tau(X)\in\chi^{-1}\{X\}$. Applying Lemma~\ref{L:prespectra} again, we see that
$\chi\circ \Psi\circ{\chi}^{-1}(\{X\})=\chi\circ \Psi\circ\tau(X)$, i.e., $G(X)=\chi\circ \Psi\circ\tau(X)$. Since each of the maps 
$\chi, \Psi, \tau$ are holomorphic, the 
claim follows.
\smallskip

\noindent{\bf Claim.} $G\in Aut\big(\Sigma^n(\Omega)\big)$.
\smallskip

\noindent{To see this, consider $H:\Sigma^n(\Omega)\lrarw\Sigma^n(\Omega)$ defined by
$H(X):=\chi\circ \Psi^{-1}\circ{\chi}^{-1}(\{X\})$ for all $X\in \Sigma^n(\Omega)$. Exactly the same argument as above shows that
$H(X)=\chi\circ \Psi^{-1}\circ\tau(X)$. From this and that $\chi\circ\tau=\mathbb{I}$ on $\Sigma^n(\Omega)$,
we get that $G\circ H=\mathbb{I}, \ H\circ G=\mathbb{I}$. This establishes that $G$ as defined in \eqref{E:autsym} 
is a holomorphic automorphism of $\Sigma^n(\Omega)$.}
\smallskip

Since $Aut(\Sigma^n(\Omega))$ satisfies the property $(\sf{P})$, as in the statement of 
Theorem~\ref{T:holautmatdom}, it is not difficult to see that there exists $\psi$, a 
holomorphic automorphism of $\Omega$, such that:
\[
G\circ\pi_n(z_1,\dots,z_n):=\pi_n\big(\psi(z_1),\dots,\psi(z_n)\big)
\]
for all $(z_1,\dots,z_n)\in\Omega^n$.
\smallskip

Now let $A\in\Sn$ be given and suppose $(\lam_1,\dots,\lam_n)$ and $(\mu_1,\dots,\mu_n)$ is a list of eigenvalues of $A$ and $\Psi(A)$ respectively,
repeated
according to their multiplicity as a zeros of characteristic polynomial.
Then from the definition of $G$ and $\chi$, we have $G\circ\pi_n(\lam_1,\dots,\lam_n)
=\pi_n(\mu_1,\dots,\mu_n)$. On the other hand from the equation above we get
 $G\circ\pi_n(\lam_1,\dots,\lam_n)=\pi_n(\psi(\lam_1),\dots,\psi(\lam_n))$. This implies 
 $\pi_n(\psi(\lam_1),\dots,\psi(\lam_n))=\pi_n(\mu_1,\dots,\mu_n)$. This in particular implies that
 $\sigma(\Psi(A))=\psi(\sigma(A))$. Since $A$ is arbitrary, this establishes the conclusion of our theorem.
\end{proof}

\subsection{A few more Preliminaries}
In this subsection, we shall gather a few more tools that are crucial to our proof of 
Theorem~\ref{T:holpropmatdom}. We shall present our proof of Theorem~\ref{T:holpropmatdom} in 
the next subsection.
\smallskip

Recall the construction of the function $f$ in the proof of Lemma~\ref{L:prespectra}. 
In particular, it shows that for any $A\in\Sn$ there exists an entire function $f$ into $\Sn$ such that $f(1)=A$ and $f(0)=D$,
where $D$ is a diagonal matrix such that $\chi(D)=\chi(A)$.
Using this property, we shall prove a proposition regarding the
pluri-complex Green function for the domain $\Sn$. Before we do this, let us first recall the pluri-complex Green function for a domain in $\C^n$.
\smallskip

Let $D\subset\C^n$ be a domain and let $P:=\{(p_j,m_j)\in D\times\mathbb{R}_{+}
: 1\leq j\leq N\}$ be a set of poles with $p_j\neq p_k$ when  $j\neq k$. Following Lelong
\cite{lel:fdplugre89}, we define the {\em pluri-complex Green function with poles in $P$} by:
\begin{align*}
g_{D}(P\,;\,w):=&\sup\big\{\nu(w): \text{$\nu\in PSH\big(D,[-\infty, 0)\big)$ and such that $\nu(z)-m_j\log||z-p_j||$}\\
                         & \text {is bounded from above in a neighborhood of $p_j$, $j=1,\dots, N$}\big\}
\end{align*}
Here, $PSH\big(D,[-\infty, 0)\big)$ denotes the set of all negative pluri-subharmonic functions on $D$.
In case $m_j=1$ for all $j$, we shall write $g_D(p_1,\dots,p_N\,;\bcdot)$ in place of $g_D(P\,;\,\bcdot)$.
Now we can state and prove the proposition alluded to in the previous paragraph regarding the 
pluri-complex Green function for $\Sn$.

\begin{proposition}\label{P:greentwo}
Let $A,B\in\Sn$. Then for any $D$, a diagonal matrix with $\chi(D)=\chi(B)$ we have:
\[
g_{\Sn}(A\,;B)=g_{\Sn}(A\,; D).
\]
\end{proposition}
\begin{proof}
As discussed above, let $f:\C\lrarw\Sn$ be a holomorphic map such that $f(0)=D$ and $f(1)=B$. Consider 
now $u:\C\lrarw [-\infty, 0)$ defined by $u(\zt):=g_{\Sn}(A\,;f(\zt))$. Then $u$ is a bounded subharmonic function
defined on $\C$. Hence $u$ has to be constant. In particular, $g_{\Sn}(A\,;B)=g_{\Sn}(A\,;f(1))=g_{\Sn}(A\,;f(0))=
g_{\Sn}(A\,;D)$.
\end{proof}

Let $\Omega\subset\C$ be a hyperconvex domain. By definition, 
there is a negative subharmonic exhaustion function $u$ on $\Omega$. In particular, $u$ 
is a non-constant, bounded above, subharmonic function on $\Omega$. This implies if we let $E=\C\setminus\Omega$ then $E\neq\emptyset$.
Moreover, $E$ cannot be a polar set and hence has to be uncountable (see, for instance, \cite[Chapter~3]{Ransford:pottheocomp}).
\smallskip

From the above discussion we see that given a hyperconvex domain $\Omega\subset\C$, and 
a positive integer $n,\,n\geq 2$, $\#(\C\setminus\Omega)\geq 2n$. Hence, $\Sigma^n(\Omega)$
is Kobayashi hyperbolic for every hyperconvex domain $\Omega$. We also wish to state that 
Zwonek in the article \cite[Proposition~11]{wz:ftpspcm20} proved that $\Sigma^n(\Omega)$ is 
hyperconvex if and only if $\Omega$ is hyperconvex. Putting all this together what we have is that 
if $\Omega$ is hyperconvex then for each $n\in\nat, n\geq 2$, the domains $\Sigma^n(\Omega)$
are Kobayashi hyperbolic and hyperconvex. 
\smallskip

Now we are in a position to present our proof of Theorem~\ref{T:holpropmatdom}.
\smallskip

\subsection{The proof of Theorem~\ref{T:holpropmatdom}}\label{SS:proofholpropmatdom}

\begin{proof}
We consider the function $G:\Sigma^n(\Omega)\lrarw\Sigma^n(\Omega)$ defined by
$G(X)=\chi\circ\Psi\circ\chi^{-1}\{X\}$. From the discussion above, we know 
that $\Omega$ satisfies the hypothesis as in Lemma~\ref{L:prespectra}. Hence by Lemma~\ref{L:prespectra},
$G$ is well defined. Exactly proceeding as in the proof of Theorem~\ref{T:holautmatdom},
we also see that $G$ is holomorphic.
\smallskip

\noindent{{\bf Claim.}} $G$ is a proper holomorphic self-map of $\Sigma^n(\Omega)$. 
\smallskip

\noindent{To establish 
this it will be sufficient to prove that if $\{X_\nu\}\subset\Sigma^n(\Omega)$ is a sequence,
having no limit points in $\Sigma^n(\Omega)$, then $\{G(X_\nu)\}$ has no limit points 
in $\Sigma^n(\Omega)$. Assume, on the contrary, a sequence $\{X_\nu\}\subset\Sigma^n(\Omega)$,
having no limit points in $\Sigma^n(\Omega)$, such that $\{G(X_\nu)\}$ has a limit point in 
$\Sigma^n(\Omega)$. This implies that there is a subsequence of $\{G(X_\nu)\}$, that we continue to
denote with $\{G(X_\nu)\}$, such that $\{G(X_\nu)\}$ converges to a point $X_0\in\Sigma^n(\Omega)$.}
\smallskip

Recall the map $\tau:\Sigma^n(\Omega)\lrarw\Sn$ as in the proof of Theorem~\ref{T:holautmatdom}.
Let $\{\lam_{1,\nu},\dots,\lam_{n,\nu}\}$ be a list of eigenvalues of $\Psi(\tau(X_\nu))$ repeated according
to their multiplicity as a zero of the characteristic polynomial of $\Psi(\tau(X_\nu))$. Then, using the 
property that $\chi\circ\tau\equiv\mathbb{I}$, we have:
\begin{equation}\label{E:eq1}
 G(X_\nu)=\chi\big(\Psi(\tau(X_\nu))\big)=\chi\big(\text{diag}[\lam_{1,\nu},\dots,\lam_{n,\nu}]\big)=
 \pi_n\big(\lam_{1,\nu},\dots,\lam_{n,\nu}\big)
 \end{equation}
where  $\text{diag}[\lam_{1,\nu},\dots,\lam_{n,\nu}]$ denotes the diagonal matrix with entries
$\lam_{j,\nu}$. Now using the properness of 
$\pi_n\big{|}_{\Omega^n}:\Omega^n\lrarw\Sigma^n(\Omega)$\,--\,and that $\{G(X_\nu)\}$ converges to $X_0\in\Sigma^n(\Omega)$\,--\,there exists a subsequence
of $\big\{\Lambda_\nu\in\Omega^n: \Lambda_\nu=(\lam_{1,\nu},\dots,\lam_{n,\nu})\big\}$, which we
continue to denote by $\{\Lambda_\nu\}$, such that $\{\Lambda_\nu\}$ converges to 
$\Lambda_0=(\lam_{1,0},\dots,\lam_{n,0})\in\Omega^n$. This, owing to equation~\eqref{E:eq1} above, implies:
\[
 \pi_n\big(\lam_{1,0},\dots,\lam_{n,0}\big)=
  \chi\big(\text{diag}[\lam_{1,0},\dots,\lam_{n,0}]\big)=X_0.\vspace{0.2cm}
  \]
  
Let $A\in\Sn$ be such that $A$ is not a critical value of $\Psi$. Let $N$ be the multiplicity of $\Psi$
and suppose $\big\{B_1,\dots,B_N\big\}=\Psi^{-1}\{A\}$.   
 The upper semicontinuity of $g_{\Sn}$ implies 
\begin{equation}\label{E:upsem}
g_{\Sn}\big(A\,;\,\text{diag}[\lam_{1,0},\dots,\lam_{n,0}]\big)
\geq \lim_{\nu\to\infty}g_{\Sn}\big(A\,;\,\text{diag}[\lam_{1,\nu},\dots,\lam_{n,\nu}]\big).
\end{equation}
Now by Proposition~\ref{P:greentwo}, and the behaviour of Green function under proper holomorphic 
mappings (see \cite[Theorem~1.2]{ae&wz:invplurigreen98}) we get:
\begin{align}\label{E:propgreen}
g_{\Sn}\big(A\,;\,\text{diag}[\lam_{1,\nu},\dots,\lam_{n,\nu}]\big)
&=g_{\Sn}\big(A\,;\,\Psi(\tau(X_\nu))\big)\nonumber\\
&=g_{\Sn}\big(B_1,\dots,B_N\,;\,\tau(X_\nu)\big)\nonumber\\
&\geq\sum_{j=1}^Ng_{\Sn}\big(B_j\,;\,\tau(X_\nu)\big)
\geq \sum_{j=1}^N g_{\Sigma^n(\Omega)}\big(\chi(B_j)\,;\,X_\nu\big).
\end{align}
Since $\Sigma^n(\Omega)$ is hyperconvex and $\{X_\nu\}$ is a sequence that does not have a 
limit point in $\Sigma^n(\Omega)$, we have 
$\lim_{\nu\to\infty}g_{\Sigma^n(\Omega)}\big(\chi(B_j)\,;\,X_\nu\big)\to 0$ for every $j, 1\leq j\leq N$;
see e.g. \cite{lel:fdplugre89}.
From this, inequality~\eqref{E:propgreen} and \eqref{E:upsem} it follows that 
$g_{\Sn}\big(A\,;\,\text{diag}[\lam_{1,0},\dots,\lam_{n,0}]\big)=0$, which is a contradiction. Hence
our assumption that $\{G(X_\nu)\}$ has a subsequence that converges in $\Sigma^n(\Omega)$ is false.
This establishes that $G$ is a proper holomorphic self-map of $\Sigma^n(\Omega)$.
\smallskip

Now since the proper holomorphic maps on $\Sigma^n(\Omega)$ satisfy the property $(\sf{P})$, there exists $\psi$, a proper holomorphic 
self-map of $\Omega$, such that $G=\Sigma^n\psi$. Using this and proceeding similarly as in the last
paragraph in the proof of Theorem~\ref{T:holautmatdom}, we get the desired result. 
\end{proof}

\section{A brief survey of holomorphic functional calculus}\label{S:holo_fc}
A very essential part of our proofs of Theorem~\ref{T:twopointT} and Theorem~\ref{T:3pt_nec} below
is the ability, given a domain $\Omega\subset\C$ and a matrix
$A\in \Sn$, to define $f(A)$ in a meaningful way for each $f\in\hol(\Omega)$. Most readers will be aware that this is what is known as the holomorphic functional calculus.
We briefly recapitulate the holomorphic functional calculus and its basic properties in a setting
which will be relevant to our proofs in the coming sections.
\smallskip

Throughout this section, $\mathfrak{X}$ will denote a finite dimensional complex Banach space and $T$
a linear operator in $\mathcal{B}(\mathfrak{X}):=\text{the set of all bounded linear operators on} 
\ \mathfrak{X}$. The symbol $\I$ will denote the identity operator and we will interpret $T^{0}=\I$.
In what follows, we shall denote by $\C[t]$ the set of all polynomials with complex coefficients in the 
indeterminate $t$. 
Given a polynomial $P\in\C[t]$, if we write $P(t):=\sum_{i=0}^n \alpha_{i}\,t^i$ with 
$\alpha_i\in\C$, then by $P(T)$ we will mean the sum $\sum_{i=0}^n\alpha_i\,T^i$.
\smallskip

For a fix $T\in\mathcal{B}(\mathfrak{X})$ and $\lam\in\C$, we consider the set
$\big\{(\lam\I-T)^j\,:\,j\in\nat\big\}$. If $\lam\notin\sigma(T)$, we notice that $\text{Ker}(\lam\I-T)^j=\{0\}$
for each $j\in\nat$. On the other hand if $\lam\in\sigma(T)$ and if we define $V^j_\lam:=\text{Ker}(\lam\I-T)^j$
then we have $\{0\}\subseteq V^1_\lam\subseteq V^2_\lam\subseteq\dots\subseteq V^j_\lam\subseteq\dots$.
Since $\mathfrak{X}$ is finite dimensional, there is a $k$, $k\leq \text{dim}(\mathfrak{X})$, such that $V^k_\lam=
V^{k+1}_\lam=V^j_\lam$ for all $j\geq k+1$.

\begin{definition}\label{D:index}
Let $T\in\mathcal{B}(\mathfrak{X})$ and let $\lam\in\C$. Then the {\em index of $\lam$},
$m(\lam)$, is defined by:
\[
m(\lam):=\min\big\{j\in\nat: \text{Ker}(\lam\I-T)^j=\text{Ker}(\lam\I-T)^{j+1}\big\}.
\]
\end{definition}
\noindent{Notice that $m(\lam)=0$ if and only if $\lam\notin\sigma(T)$.
A question arises at this point: given polynomials $P,\,Q\in\C[t]$, and $T\in\mathcal{B}(\mathfrak{X})$,
when $P(T)=Q(T)$? We state a very important result that among other things primarily 
provides answer to this question:}
\begin{result}\cite[Chapter~7, Section~1]{dunfordsch:Linopera88}\label{Res:keyresholfunc}
Let $P,\,Q\in\C[t]$, and let
$T\in\mathcal{B}(\mathfrak{X})$. Then we have $P(T)=Q(T)$ if and only if each $\lam\in\sigma(T)$
is a zero of $P-Q$ of order $m(\lam)$.
\end{result}
\noindent{We refer the reader to \cite[Chapter~7, Section~1]{dunfordsch:Linopera88} for a proof of this result. 
The proof as given in \cite{dunfordsch:Linopera88} also shows that $m(\lam)$ is the multiplicity of $\lam$
as a zero of the minimal polynomial of $T$.}
\smallskip

Given $T\in\mathcal{B}(\mathfrak{X})$, let $\mathcal{F}(T)$ be the set of all holomorphic functions in a 
neighbourhood of $\sigma(T)$. Given $f\in\mathcal{F}(T)$, we define:
\[
f(T):=P(T), \ \text{where $P\in\C[t]$ with $f^{(j)}(\lam)=P^{(j)}(\lam)$, $0\leq j\leq m(\lam)-1$,
 $\forall\lam\in\sigma(T)$}.
 \] 
Here, $f^{(j)},\,P^{(j)}$ denotes the $j$-th derivative of $f$ and $P$ respectively.
 It follows from Result~\ref{Res:keyresholfunc} that the definition above is unambiguous and the function 
$\Theta_T:\mathcal{F}(T)\lrarw\mathcal{B}(\mathfrak{X})$ defined by $\Theta_T(f):=f(T)$ has following 
properties: if $f,g\in\mathcal{F}(T)$ and $\alpha, \beta\in\C$ then:\vspace{0.15cm}
\begin{itemize}
\item $\alpha f+\beta g\in\mathcal{F}(T)$, and $\Theta_T(\alpha f+\beta g)=\alpha\,\Theta_T(f)+
\beta\,\Theta_T(g)$, 
\vspace{0.1cm}

\item $fg\in\mathcal{F}(T)$ and $\Theta_T(fg)=\Theta_T(f)\,\Theta_T(g)$,
\vspace{0.1cm}

\item $\sigma(\Theta_T(f))=\sigma(f(T))=f(\sigma(T))$.\vspace{0.15cm}
\end{itemize}
The last property above is called the Spectral Mapping Property in literature. We also note that from the second property above, it follows that $f(T)g(T)=g(T)f(T)$ for all $f,g\in\mathcal{F}(T)$.
\smallskip

Given $p\in\C$, let $e_p(\bcdot)$ be a function that is identically equal to $1$ in a neighbourhood of $p$,
and identically equal to $0$ in a neighbourhood of each point of $\sigma(T)\cap{(\C\setminus\{p\})}$. Set 
$E(p)=e_p(T)$. Then we have following:

\begin{result}\cite[Chapter~7, Section~1]{dunfordsch:Linopera88}\label{Res:orthdeco}
Given $T\in\mathcal{B}(\mathfrak{X})$ and $p\in\C$, let $E(p)\in\mathcal{B}(\mathfrak{X})$ be as defined
above then we have:
\begin{itemize}
\item $E(p)\neq 0$ if and only if $p\in\sigma(T)$.
\vspace{0.1cm}
\item $E(p)^2=E(p)$ and $E(p)E(q)=0$ for $p\neq q$.
\vspace{0.1cm}
\item $\I=\sum_{p\in\sigma(T)} E(p)$.
\end{itemize}
\end{result}

\begin{remark}\label{Rm:orthdecoX}
If $\{\lam_1,\dots,\lam_k\}$ be an enumeration of $\sigma(T)$, and let $\mathfrak{X}_i=E(\lam_i)\mathfrak{X}$.
Result~\ref{Res:orthdeco} implies that
\[
\mathfrak{X}=\mathfrak{X}_1\oplus\dots\oplus\mathfrak{X}_k.
\]
Moreover, since $TE(\lam_i)=E(\lam_i)T$, it follows that $T\mathfrak{X}_i\subseteq\mathfrak{X}_i$, $i=1,\dots,
k$. Thus, to the decomposition of the spectrum $\sigma(T)$ into $k$ points there corresponds a direct sum 
decomposition of $\mathfrak{X}$ into $k$ invariant subspaces of $T$. Thus the study of the action of $T$ on $\mathfrak{X}$ may be reduced to the study of the action of $T$ on each of the subspaces $\mathfrak{X}_i$.
\end{remark}

\begin{remark}\label{Rm:nilpotency}
If we restrict $T$ on $\mathfrak{X_i}$ and write $T=\lam_i\,\I+(T-\lam_i\,\I)$ on $\mathfrak{X_i}$. Then
since the holomorphic function $(t-\lam_i)^{m(\lam_i)}e_{\lam_i}(t)$ has a zero of order $m(\lam_i)$ at each point
of $\sigma(T)$, we have $(T-\lam_i\I)^{m(\lam_i)}\,E(\lam_i)=0$. Clearly 
$(T-\lam_i\I)^{m(\lam_i)-1}\,E(\lam_i)\neq 0$. This shows that restricted to $\mathfrak{X}_i$,
the operator $T-\lam_i\,\I$ is a nilpotent operator of order $m(\lam_i)$. Thus,
in each space $\mathfrak{X_i}$, the operator  $T$ is the sum of a scalar multiple $\lam_i\,\I$ of the
identity and a nilpotent operator $T-\lam_i\I$ of order $m(\lam_i)$.
\end{remark}

We now state a result that gives an explicit formula for computing $f(T)$ in terms of the projections
$E(\lam)$, $\lam\in\sigma(T)$.
\begin{result}\label{Res:fomlfuncal}
Let $T\in\mathcal{B}(\mathfrak{X})$ and let $f\in\mathcal{F}(T)$ be given. Then
\begin{equation}\label{E:fomlfuncal}
f(T)=\sum_{\lam\in\sigma(T)}\sum_{j=0}^{m(\lam)-1}
\intf{}{}{(T-\lam\I)^j}{j!}\,f^{(j)}(\lam)E(\lam).
\end{equation}
\end{result}
\noindent{The formula \eqref{E:fomlfuncal} follows immediately from the properties of the function $\Theta_{T}$ described 
above; once we observe the following: given $f\in\mathcal{F}(T)$, consider the function
$g\in\mathcal{F}(T)$ defined by:
\[
g(t):=\sum_{\lam\in\sigma(T)}\sum_{j=0}^{m(\lam)-1}
\intf{}{}{(t-\lam)^j}{j!}\,f^{(j)}(\lam)e_{\lam}(t)\, ,
\]
then we have $f^{(i)}(\lam)=g^{(i)}(\lam),\,i\leq\,m(\lam)-1$, for $\lam\in\sigma(T)$.}

\section{The proof of Theorem~\ref{T:twopointT}}\label{S:twopointT}
We shall present our proof of Theorem~\ref{T:twopointT} in this section. 
Before we present our proof, we shall need certain complex analytic tools
related to the spectral unit ball. We first discuss them in the next subsection. 

\subsection{Complex analytic properties of the spectral unit ball}\label{ss:capsub}
For $n \in \Z_+$, recall the spectral unit ball, $\OM_n \subset \C^{n^2}$, is the
collection of all matrices $A \in M_n(\C)$ whose
spectrum $\sigma(A)$ is contained in $\D$. We also recall the definition of spectral
radius $\rho$ of a matrix $A$ defined by $\rho(A):=\max\big\{|\,\lam\,|:\lam\in\sigma(A)\big\}$.
We first state the result:

\begin{result}[Janardhanan, \cite{jj:slpgf20}]\label{Res:jjslpgf}
The spectral unit ball, $\OM_n$, is an unbounded, balanced, pseudo-convex domain with Minkowski 
function given by the spectral radius $\rho$.
\end{result}
\noindent{It is a fact that the Minkowski function of a balanced pseudo-convex domain is pluri-subharmonic
(see \cite[Appendix B.7.6]{JP13:jarnicki2013invariant}). Hence, 
it follows from this result that $\rho|_{\OM_n}$ is
pluri-subharmonic. The pluri-subharmonicity of spectral radius function also follows from another
important result due to Vesentini \cite{vst:shSprd68}.} This latter result regarding the spectral radius function
is for a general Banach algebra.
\smallskip

We now state an important lemma for holomorphic functions in $\hol(\D,\,\Omega_n)$. 
This lemma could be considered as a generalization of 
the Schwarz lemma for holomorphic functions in $\hol(\D,\,\D)$. 
\begin{lemma}\label{L:fl2}
Let $F\in\hol(\D,\,\Omega_n)$ be such that $F(0)=0$. Then there exists
$G\in\hol(\D,\,\overline{\Omega}_n)$ such that $F(\zt)\,=\,\zt\,G(\zt)$ for all $\zt\in\D$.
In particular, we have $\rho(F(\zt))\leq |\zt|$ for all $\zt\in\D$.
\end{lemma}
\noindent{The lemma is a consequence of the fact that $\rho|_{\OM_n}$ is
pluri-subharmonic. We do not wish to present a proof of the above lemma here;
we refer the interested reader to \cite[Lemma~4.3]{chandel:3spintp20} for a proof.
We also wish to state another another lemma which is a consequence of the 
pluri-subharmonicity of the spectral radius function. }
\begin{lemma}\label{L:fl1}
Let $\Phi\in\hol(\D,\,\overline{\Omega}_n)$ be such that there exists a
$\theta_0\in\mathbb{R}$ and $\zt_0\in\D$ satisfying $e^{i\theta_0}\in\sigma(\Phi(\zt_0))$.
Then $e^{i\theta_0}\in\sigma(\Phi(\zt))$ for all $\zt\in\D$.
\end{lemma}
\noindent{Lemma~\ref{L:fl1} is not needed in the proof of Theorem~\ref{T:twopointT} although it will be 
an important tool in the proof of Theorem~\ref{T:3pt_nec}. We stated it here since it follows from the 
pluri-subharmonicity of spectral radius function; see e.g. \cite[Section~4]{chandel:3spintp20}. We are now 
ready to present our proof of Theorem~\ref{T:twopointT}.}
\smallskip

\subsection{The proof of Theorem~\ref{T:twopointT}}\label{SS:twopointT}
\begin{proof}
Let $F\in\hol(\D,\,\Sn)$ be such that $F(\zt_j)=W_j$, $j=1,2$. For each $k\in\{1,2\}$, consider
$\Phi_k\in\hol(\D,\,\Sn)$ defined by:
\begin{equation}\label{E:preautcom}
  \Phi_k(\zt)=F\circ\psi^{-1}_{k}(\zt) \ \ \  \forall\zt\in\D.\vspace{0.1cm}
   \end{equation}
Here, $\psi_k$ is
the automorphism $\psi_k(\zt):=(\zt-\zt_k)(1-\overline{\zt_k}\zt)^{-1}$, $\zt\in\D$, of $\D$. 
Then $\Phi_k(0)=W_k$ and $\Phi_k(\psi_{k}(\zt_j))=W_j$, $j\neq k$. Now for an arbitrary 
but fixed $z\in\Omega\setminus\sigma(W_k)$, consider $B_{k,\,z}\in\hol(\Omega,\,\D)$ defined by:
\begin{equation}\label{E:anablash}
 B_{k,\,z}(\bcdot):=\prod_{\lambda\in\sigma(W_k)}G_{\Omega}(\lambda,\,z;\,\bcdot)^{m(\lambda)}.
 \vspace{0.1cm}
\end{equation}
Observe that $ B_{k,\,z}(\bcdot)=B(W_k,\,z;\,\bcdot)$ is as in Definition~\ref{Def:genminblashprod}.
\smallskip

As $B_{k,\,z}\in\hol(\Omega)$, it induces\,---\,via the holomorphic functional calculus\,---\,a map (which we continue to denote by $B_{k,\,z}$) from $\Sn$ to $M_n(\C)$. The Spectral Mapping Theorem tells us that 
$\sigma(B_{k,\,z}(X))=B_{k,\,z}(\sigma(X))\subset\D$ for every $X\in\Sn$. Hence 
$B_{k,\,z}(X)\subset\Omega_n$ for every $X\in\Sn$.
\smallskip

\noindent{\bf Claim.} $B_{k,\,z}(\Phi_k(0))=0$.
\smallskip

\noindent To see this we write:\vspace{0.2cm}
\[
B_{k,\,z}(\zt)=\Big(\prod\nolimits_{\lambda\in\sigma(W_k)}(\zt-\lambda)^{m(\lambda)}\Big)\,g_{z}(\zt)
\vspace{0.2cm}
\]
for some $g_z\in\hol(\Omega)$ and for every $\zt\in\Omega$. Hence, since\,---\,by the holomorphic 
functional calculus\,---\,the assignment $f\longmapsto f(\Phi_k(0)), f\in\hol(\Omega)$, is multiplicative, 
as discussed in Section~\ref{S:holo_fc}, we get
\[
 B_{k,\,z}(\Phi_k(0))=\Big(\prod\nolimits_{\lambda\in\sigma(W_k)}(\Phi_k(0)-\lambda\,\mathbb{I})^{m(\lambda)}\Big)\,g_z(\Phi_k(0)).\vspace{0.2cm}
\]
Now since the minimal polynomial for $W_k=\Phi_k(0)$ is given by $\prod\nolimits_{\lambda\in\sigma(W_k)
}(t-\lambda)^{m(\lambda)}$, we see that the product term in the right hand side of the above equation is
zero, whence the claim.
\smallskip

Consider the map $\Psi_{k,\,z}$ defined by:
\[
\Psi_{k,\,z}(\zt):=B_{k,\,z}\circ\Phi_k(\zt), \ \zt\in\D.
\]
It is a fact that $\Psi_{k,\,z}\in\hol(\D)$. Moreover, from the above claim and the discussion just before it,
we have $\Psi_{k,\,z}\in\hol(\D,\,\Omega_n)$ with $\Psi_{k,\,z}(0)=0$. By Lemma~\ref{L:fl2}, we get that
\[
\rho(\Psi_{k,\,z}(\zt))\leq |\zt| \ \ \ \forall\zt\in\D.
\]
Now from the definition of $\Psi_{k,\,z}$ and from the Spectral Mapping Theorem we get 
$\sigma(\Psi_{k,\,z}(\zt))=\sigma(B_{k,\,z}(\Phi_k(\zt)))=B_{k,\,z}(\sigma(\Phi_k(\zt)))$. This together with the above equation 
gives us:
\[
|\,B_{k,\,z}(\mu)\,|\leq |\,\zt\,| \ \ \ \forall\zt\in\D \ \text{and} \ \mu\in\sigma(\Phi_k(\zt)).
\]
We put $\zt=\psi_{k}(\zt_j)$ in the above equation. Then, since $\Phi_k(\psi_{k}(\zt_j))=W_j$, by the 
above equation and \eqref{E:anablash}, we get:
\begin{equation}\label{E:prodcaraext}
|\,B_{k,\,z}(\mu)\,|\,=\,\prod_{\lambda\in\sigma(W_k)}|\,G_{\Omega}(\lambda,\,z;\,\mu)\,|^{m(\lambda)} \leq |\,\psi_{k}(\zt_j)\,|
\ \ \ \forall\mu\in\sigma(W_j).
\end{equation}
Since $z$ is arbitrary, for an arbitrary but fixed $\mu\in\sigma(W_j)$, we can take $z=\mu$ in the above 
equation. This with the observation that $G_{\Omega}(\lambda,\,\mu;\,\mu)=C_{\Omega}(\lambda,\mu)$
gives us that:\vspace{0.2cm}
\begin{equation}\label{E:oneineq}
\prod_{\lambda\in\sigma(W_k)}\,{C_{\Omega}(\lam,\mu)}^{m(\lam)} \ 
\leq \ \hyper{\zt_1}{\zt_2} \ \ \ \forall\mu\in\sigma(W_j).
\vspace{0.2cm}
\end{equation}
Interchanging the roles of $j$ and $k$, in the above discussion will give an inequality 
of the form \eqref{E:oneineq} with $j$ and $k$ interchanged.
The inequality \eqref{E:2SchwarzIneq} will follow from these two inequalities. 
\end{proof}

\section{Minimal polynomials under holomorphic functional calculus}\label{S:compminipoly}
In this section, we develop the key matricial tool needed in establishing Theorem~\ref{T:3pt_nec}, which is 
the computation of the minimal polynomial for $f(A)$, given $f\in\hol(\Omega)$ and
$A\in\Sn$, $n\geq 2$. 
This is the content of Theorem~\ref{T:minpo_holo_func_anal} below.
In what follows, given integers $p < q$, $\intgR{p}{q}$
will denote the set of
integers $\{p, p+1,\dots, q\}$. Given $A\in M_{n}(\C)$, we will
denote its minimal polynomial by $\minpo{A}$. We begin with a lemma:
\smallskip

\begin{lemma}[Chandel, \cite{chandel:3spintp20}]\label{L:minmo_lincomb_nilpo}
Let $(\alpha_0,\alpha_1,\dots,\alpha_{n-1})\in\C^{n}$, $n\geq 2$. Let 
$A\,=\,\sum_{j=0}^{n-1}\alpha_jN^{j}$, where $N$ is the
nilpotent operator of degree $n$. Then the minimal polynomial for $A$ is given by:
\begin{equation}\label{E:minpoeq_lincomb_nilpo}
\minpo{A}(t)\,=\,(t-\alpha_{0})^{[(n-1)/{l(\alpha_1,\alpha_2,\dots,\alpha_{n-1})}]+1}.
\end{equation}
\end{lemma}
\noindent{Here, $[\bcdot]$ denotes the greatest integer function and $l(\alpha_1,\alpha_2,\dots,\alpha_{n-1})$ is defined by:
\[
l(\alpha_1,\alpha_2,\dots,\alpha_{n-1})\,:=\,
\begin{cases}
n, &\text{if $\alpha_j = 0 \ \forall j\in\intgR{1}{n-1}$},\\
\min\{j\in \intgR{1}{n-1}\,:\,\alpha_j\neq 0\}, &\text{otherwise}.
\end{cases}
\]}
The reader is referred to \cite[Lemma~3.1]{chandel:3spintp20} for a proof of this lemma.
\smallskip

Given $a\in \C$ and $g$ a holomorphic function in a neighbourhood of $a$,
$\mathsf{ord}_{a}g$ will denote the order of vanishing of $g$ at $a$, as defined after the statement 
of Theorem~\ref{T:3pt_nec}. We now present the main result of this section.

\begin{theorem}\label{T:minpo_holo_func_anal}
Let $A\in\Sn$, $n\geq 2$, and let $f\in\hol(\Omega)$ be a non-constant function. Suppose that the minimal
polynomial for $A$ is given by
\[
\minpo{A}(t)\,=\,\prod_{\lambda\in\sigma(A)}(t-\lambda)^{m(\lambda)}.
\]
Then the minimal polynomial for $f(A)$ is given by
\[
\minpo{f(A)}(t)\,=\,\prod_{\nu\in f(\sigma(A))}(t-\nu)^{k(\nu)},
\]
where, $k(\nu)=\max\left\{\intf{[}{]}{m(\lambda)-1}
{\mathsf{ord}_{\lambda}{f'}+1}+1:\,\,\lambda \in \sigma(A)\cap f^{-1}\{\nu\}\right\}.
$
\end{theorem}

\begin{proof}
Let $\lam\in\sigma(A)$ and let $E(\lam)$ be the projection operator as defined in Section~\ref{S:holo_fc}
with $T=A$.
Then by Remark~\ref{Rm:orthdecoX} we know that:
\begin{equation}\label{E:orthdecoX}
 \mathfrak{X}=\oplus_{\lam\in\sigma(A)}\mathfrak{X}_\lam, \ \ \ \text{where $\mathfrak{X}_\lam=E(\lam)
 \mathfrak{X}, \ \lam\in\sigma(A)$.}
\end{equation}
Note that $\mathfrak{X}=\C^n$, but it really does not matter here.
Let $x\in\mathfrak{X}$ and write $x=\oplus_{\lam\in\sigma(A)}x_{\lam}$, where $x_\lam=E(\lam)x$. Then
by Result~\ref{Res:fomlfuncal} we have:
\begin{align*}
 f(A)x&=\sum_{\lam\in\sigma(A)}\Bigg(\sum_{j=0}^{m(\lam)-1}
 \intf{}{}{(A-\lam\I)^j}{j!}\,f^{(j)}(\lam)E(\lam)\Bigg)\,x \\ 
 &=\sum_{\lam\in\sigma(A)}\Bigg(\sum_{j=0}^{m(\lam)-1}
 \intf{}{}{(A-\lam\I)^j}{j!}\,f^{(j)}(\lam)\Bigg)\,x_\lam .
\end{align*}
Now consider the operators $A_\lam : \mathfrak{X}_\lam\lrarw\mathfrak{X}$ 
defined by:
\[
  A_\lam:=\sum_{j=0}^{m(\lam)-1}\intf{}{}{(A-\lam\I)^j}{j!}\,f^{(j)}(\lam).
\]
Since $A\mathfrak{X}_\lam\subseteq\mathfrak{X}_\lam$ for all $\lam\in\sigma(A)$, $A_\lam$ leaves
$\mathfrak{X}_\lam$ invariant; i.e. $A_\lam(\mathfrak{X}_\lam)\subseteq\mathfrak{X}_\lam$. Hence
\vspace{0.1cm}
\begin{equation}\label{E:f(A)deco}
  f(A)\,x=\oplus_{\lam\in\sigma(A)}A_{\lam}\,x_\lam, \ \text{where $x=\oplus_{\lam\in\sigma(A)}\,x_{\lam}$
  is as in \eqref{E:orthdecoX}.}\vspace{0.1cm}
\end{equation}
This establishes that $f(A)=\oplus_{\lam\in\sigma(A)}\,A_\lam$.
\smallskip

Focussing our attention to $A_\lam$, we first observe that\,--\,by Remark~\ref{Rm:nilpotency}\,--\,the operator
$A-\lam\I$, $\lam\in\sigma(A)$, restricted to $\mathfrak{X}_\lam$ is a nilpotent 
operator of degree $m(\lam)-1$.
Hence by Lemma~\ref{L:minmo_lincomb_nilpo}, we have $\minpo{A_\lam}(t):=(t-f(\lam))^{\nu}$,
where\vspace{0.1cm}
\[
\nu\,=\,\intf[]{m(\lam)-1}{l(f'(\lambda),f''(\lambda),\dots,f^{(m(\lam)-1)}(\lambda))}+1.\vspace{0.1cm}
\]
If $\mathsf{ord}_{\lambda}{f'}\leq m(\lam)-2$,
then $l(f'(\lambda),f''(\lambda),\dots,f^{(m(\lam)-1)}(\lambda))\,=\,\mathsf{ord}_{\lambda}{f'}+1$,
else $\mathsf{ord}_{\lambda}{f'}+1>(m(\lam)-1)$ and $l(f'(\lambda),f''(\lambda),\dots,f^{(m(\lam)-1)}(\lambda))>(m(\lam)-1)$.
In both cases we have:\vspace{0.1cm}
\[
\intf[]{m(\lam)-1}{l(f'(\lambda),f''(\lambda),\dots,f^{(m(\lam)-1)}(\lambda))}\,=\,
\intf[]{m(\lam)-1}{\mathsf{ord}_{\lambda}{f'}+1}.\vspace{0.1cm}
\]
From the last two expressions $\minpo{A_\lam}(t)
=(t-f(\lam))^{\intf[]{m(\lam)-1}{\mathsf{ord}_{\lambda}{f'}+1}}.$
\smallskip

Let us rewrite \eqref{E:f(A)deco} in another way:
\[
f(A)=\oplus_{\nu\in f(\sigma(A))}\,B_\nu, \ \ \ \text{where} \ B_\nu=\oplus_{f(\lam)=\nu}\,A_\lam
\]
The minimal polynomial for $f(A)$, $\minpo{f(A)}$, is the least common multiple of $\minpo{B_\nu}$,
$\nu\in f(\sigma(A))$. Notice that the polynomials $\minpo{B_\nu}$, $\nu\in f(\sigma(A))$ are relatively
prime to each other. Hence
\[
\minpo{f(A)}(t)=\prod_{\nu\in f(\sigma(A))}\,\minpo{B_\nu}(t)
\]
Now the minimal polynomial for $B_\nu$ is the least common multiple of minimal polynomials of 
$A_\lam$ such that $f(\lam)=\nu$, $\lam\in\sigma(A)$. It is easy to see (using the expression for
$\minpo{A_\lam}$) that
\[
\minpo{B_\nu}(t)=(t-\nu)^{k(\nu)}, \ \text{where $k(\nu)=
\max\Bigg\{\intf[]{m(\lam)-1}{\mathsf{ord}_{\lambda}{f'}+1}\,:\,\lam\in f^{-1}\{\nu\}\cap\sigma(A)\Bigg\}$}
\]
From the last two expressions, we get the desired result.
\end{proof}
\begin{remark}
Theorem~\ref{T:minpo_holo_func_anal} is a generalization of Theorem~3.4 in \cite{chandel:3spintp20} 
which dealt with the 
case $\Omega=\D$. The proof of Theorem~3.4 as presented in \cite{chandel:3spintp20} is
very specific to the unit disc.
This is because the proof of Lemma~3.2\,--\,which is a crucial tool in establishing Theorem~3.4 in 
\cite{chandel:3spintp20}\,--\,exploits a property of a holomorphic function in the unit disc; namely:
every holomorphic
function in $\D$ has a power series representation on $\D$.
This latter property is very specific to the unit disc.
\end{remark}

\section{The proof of Theorem~\ref{T:3pt_nec}}\label{S:3pint}
In this section, we shall present our proof of Theorem~\ref{T:3pt_nec}. After the proof, we shall also present an 
observation that, among other things, shows that when $\Omega=\D$, Theorem~1.4 in \cite{chandel:3spintp20} is the same as our theorem, in particular, \cite[Theorem~1.4]{chandel:3spintp20} is a special case of 
Theorem~\ref{T:3pt_nec}.
\smallskip

Before we present our proof, we wish to restate the inequality \eqref{E:2SchwarzIneq} when
$\Omega=\D$, so that it is easy to apply in our proof of Theorem~\ref{T:3pt_nec}. First of all, as mentioned before, in this case $C_{\Omega}(z_1,\,z_2)=\mobi(z_1,\,z_2)$.
Now given $W_j\in\Omega_n$, $j=1,2$, consider:
\[
b_j(t)=\prod_{\lam\in\sigma(W_j)}{\intf{(}{)}{t-\lam}{1-\overline{\lam}t}^{m(j,\,\lam)}}.
\]
Here $m(j,\,\lam)$ denotes the multiplicity of $\lam$ as a zero of the minimal polynomial for 
$W_j$. The finite Blaschke product $b_j$ is called the {\em minimal Blaschke product} corresponding to $W_j$,
$j=1,2$. With this notation in hand, we can restate the inequality \eqref{E:2SchwarzIneq} as:\vspace{0.1cm}
\begin{equation}\label{E:alttwopointT}
\max\left\{\max_{\mu\in\sigma(W_2)}|\,b_1(\mu)\,|, 
\ \max_{\lambda\in\sigma(W_1)}|\,b_2(\lam)\,|\right\} \ 
\leq \ \hyper{\zt_1}{\zt_2}.\vspace{0.1cm}
\end{equation}
It is in this form that we shall use the inequality~\eqref{E:2SchwarzIneq} in our proof when $\Omega=\D$.
We now present our proof of Theorem~\ref{T:3pt_nec}.

\subsection{The proof of Theorem~\ref{T:3pt_nec}}\label{SS:3pt_nec_proof}

Let $F\in\hol(\D,\,\Sn)$ be such that $F(\zt_j)=W_j$, $j\in\{1,2,3\}$. Let 
$k\in\{1,2,3\}$ and let us denote by $B_{k,\,z}(\bcdot)$ the function $B(W_k,\,z;\,\bcdot)$,
where $z\in\Omega\setminus\sigma(W_k)$ be any arbitrary but fixed point in $\Omega$.
Now consider the function:
\[
\widetilde{F}_{k,\,z}:=B_{k,\,z}\circ F\circ\psi_k^{-1}.
\]
Here $\psi_k(\zt)={(\zt-\zt_k)}/{(1-\bar{\zt_k}\zt)}$ is an automorphism of the unit disc that maps 
$\zt_k$ to $0$. Then $\widetilde{F}_{k,\,z}\in\hol(\D,\,\OM_n)$ such that 
$\widetilde{F}_{k,\,z}(\psi_k(\zt_{L(k)}))=B_{k,\,z}(W_{L(k)}),\,\widetilde{F}_{k,\,z}(\psi_k(\zt_{G(k)}))
=B_{k,\,z}(W_{G(k)})$ and $\widetilde{F}_{k,\,z}(0)=0$. By Lemma~\ref{L:fl2}, we get
\begin{equation}\label{E:fact_aux_intp}
\widetilde{F}_{k,\,z}(\zt)=\zt\,\widetilde{G}_{k,\,z}(\zt) \ \forall\zt\in\D,\,\,\text{for some 
$\widetilde{G}_{k,\,z}\in\hol(\D,\,\overline{\Omega}_n)$}.
\end{equation}
Two cases arise:
\smallskip

\noindent{\bf Case 1.} $\widetilde{G}_{k,\,z}(\D)\subset\Omega_n$.

\noindent In view of \eqref{E:fact_aux_intp}, we have \vspace{0.1cm}
\begin{equation}\label{E:2pt_redut_intp}
\widetilde{G}_{k,\,z}\left(\psi_k(\zt_{L(k)})\right)\,=\,W_{L(k),\,k,\,z}\,\,\text{and}\,\,
\widetilde{G}_{k,\,z}\left(\psi_k(\zt_{G(k)})\right)\,=\,W_{G(k),\,k,\,z}\vspace{0.1cm}
\end{equation}
where $W_{L(k),\,k,\,z}:={B_{k,\,z}(W_{L(k)})}\big{/}{\psi_{k}(\zt_{L(k)})}$ and
$W_{G(k),\,k,\,z}:={B_{k,\,z}(W_{G(k)})}\big{/}{\psi_{k}(\zt_{G(k)})}$.
Now using the inequality \eqref{E:alttwopointT}, a necessary condition for \eqref{E:2pt_redut_intp} is
\begin{equation}\label{E:ineq_2pt_redut}
\max\left\{\sub{\eta\in\sigma\left(W_{L(k),\,k,\,z}\right)}{\max}|\,b_{G(k),\,k,\,z}(\eta)\,|,\,
\sub{\eta\in\sigma\left(W_{G(k),\,k,\,z}\right)}{\max}|\,b_{L(k),\,k,\,z}(\eta)\,|\right\}
\leq\mathcal{M}_{\D}\left(\zt_{G(k)},\,\zt_{L(k)}\right),\vspace{0.1cm}
\end{equation}
where $b_{L(k),\,k,\,z}$ and $b_{G(k),\,k,\,z}$ denote the minimal Blaschke product corresponding to the 
matrices $W_{L(k),\,k,\,z}\,,\,\,W_{G(k),\,k,\,z}$. Given the definitions of the latter matrices, we will need 
Theorem~\ref{T:minpo_holo_func_anal} to determine  $b_{L(k),\,k,\,z}$, $b_{G(k),\,k,\,z}$. By this theorem, we have
\begin{align}
b_{L(k),\,k,\,z}(t)&=\prod_{\nu\in\sigma\left(B_{k,\,z}(W_{L(k)})\right)}
\!\!\!{\intf{(}{)}{t-{\nu}/{\psi_k(\zt_{L(k)})}}{1-\overline{{\nu}/{\psi_k(\zt_{L(k)})}}t}}^{q_z(\nu,\,L(k),\,k)}\,\,
\label{E:minpo_2pt_redut1}\\ 
b_{G(k),\,k,\,z}(t)&=\prod_{\nu\in\sigma\left(B_{k,\,z}(W_{G(k)})\right)}
\!\!\!{\intf{(}{)}{t-{\nu}/{\psi_k(\zt_{G(k)})}}{1-\overline{{\nu}/{\psi_k(\zt_{G(k)})}}t}}^{q_z(\nu,\,G(k),\,k)},\label{E:minpo_2pt_redut2}
\end{align}
where $q_z(\nu,\,L(k),\,k)$ and $q_z(\nu,\,G(k),\,k)$ are as in the statement of Theorem~\ref{T:3pt_nec}.
Now if $\eta\in\sigma(W_{L(k),\,k,\,z})$ or $\eta\in\sigma(W_{G(k),\,k,\,z})$,
then $\eta=\mu/{\psi_k(\zt_{L(k)})}$ for some $\mu\in\sigma(B_{k,\,z}(W_{L(k)}))$ or
$\eta=\mu/{\psi_k(\zt_{G(k)})}$ for some $\mu\in\sigma(B_{k,\,z}(W_{G(k)}))$, respectively, and conversely.
This observation together with \eqref{E:minpo_2pt_redut2}, \eqref{E:minpo_2pt_redut1} and 
\eqref{E:ineq_2pt_redut} establishes the first part of our theorem.
\smallskip

\noindent{\bf Case 2.} $\widetilde{G}_{k,\,z}(\D)\cap\partial\,\Omega_n\not=\emptyset$.

\noindent Let $\zt_0\in\D$ be such that $e^{i\theta_z}\in\sigma(\widetilde{G}_{k,\,z}(\zt_0))$
for some $\theta_z\in\mathbb{R}$. By Lemma~\ref{L:fl1}, we have
$e^{i\theta_z}\in\sigma(\widetilde{G}_{k,\,z}(\zt))$ for every $\zt\in\D$. By \eqref{E:fact_aux_intp},
$e^{i\theta_z}\zt\in\sigma(\widetilde{F}_{k,\,z}(\zt))$. Let $\Phi\,\equiv\,F\circ\psi_k^{-1}$. Then
$\Phi\in\hol\big(\D,\,\Sn\big)$ and we have:
\[
e^{i\theta_z}\zt\in\sigma\big(B_{k,\,z}\circ\Phi(\zt)\big)=B_{k,\,z}\{\sigma(\Phi(\zt))\} \ \forall\zt\in\D,
\]
where the last equality is an application of the Spectral Mapping Theorem.
For each $\zt\in\D$, let $\omega_{\zt}\in\sigma(\Phi(\zt))$ be such that
$B_{k,\,z}(\omega_{\zt})\,=\,e^{i\theta_z}\zt$. Notice that if $\zt_1\not=\zt_2$ then
$\omega_{\zt_1}\not=\omega_{\zt_2}$, whence $E\,:=\,\{\omega_{\zt}\,:\,\zt\in\D\}$ is
an uncountable set in $\Omega$. Notice that $\omega_{\zt}$ satisfies:
\[
B_{k,\,z}(\omega_\zt)=e^{i\theta_z}\zt\,\,\,\text{and}\,\,\,
\mathrm{det}\left(\omega_{\zt}\mathbb{I}-\Phi(\zt)\right)
=0 \ \forall\zt\in\D.
\]
As $E$ is uncountable, it follows from the identity principle that the holomorphic map\linebreak
$x\longmapsto\mathrm{det}\left(x\mathbb{I}-\Phi(e^{-i\theta_z}B_{k,\,z}(x))\right)$ is identically $0$
on $\Omega$.
As $B_{k,\,z}$ maps $\Omega$ into $\D$, it follows that
\begin{equation}\label{E:blasholcor}
B^{-1}_{k,\,z}\{e^{i\theta_z}\zt\}\subset\sigma(\Phi(\zt))=\sigma\left(F\circ\psi_k^{-1}(\zt)\right) \ \forall\zt\in\D.
\end{equation}
Putting $\zt=\psi_k(\zt_{L(k)})$ and $\psi_k(\zt_{G(k)})$ respectively in \eqref{E:blasholcor}
we get 
$B^{-1}_{k,\,z}\{e^{i\theta_z}\psi_k(\zt_{L(k)})\}\subset\sigma(F(\zt_{L(k)}))
=\sigma(W_{L(k)})$ and
$B^{-1}_{k,\,z}\{e^{i\theta_z}\psi_k(\zt_{G(k)})\}\subset\sigma(F(\zt_{G(k)}))=\sigma(W_{G(k)})$.\qed
\smallskip

\begin{obs}\label{Obs:coromaintD}
When $\Omega=\D$, it is not difficult to show that 
\[
 G_{\D}(\lam,\,z;\,\zt):=V(\lam, z)\,\intf{}{}{\zt-\lam}{1-\bar{\lam}\zt}
\]
where $V(\lam, z)\in\mathbb{T}$ such that $V(\lam, z)\big((z-\lam)/(1-\bar{\lam}z)\big)=
\mathcal{M}_{\D}(\lam,\,z)$. Hence when $\Omega=\D$, and $A\in\Omega_n$ then 
$B(A,\,z;\,\bcdot)$ is of the form
\[
B(A,\,z;\,\bcdot):=R\big(\sigma(A),\,z\big)\prod_{\lam\in\sigma(A)}{
\intf(){\zt-\lam}{1-\bar{\lam}\zt}}^{m(\lam)}.
\]
Here, $R\big(\sigma(A),\,z\big)$ is a uni-modular constant. In particular, $B(A,\,z;\,\bcdot)$ at any point
$z$ is a scalar multiple of the minimal Blaschke product corresponding to $A$ by a uni-modular constant depending on $z$.
\smallskip

In this case, given $W_i\in\Omega_n$, $i\in\{1,2,3\}$ as in the statement of Theorem~\ref{T:3pt_nec}, we notice that:
\[
 \mathsf{ord}_{\lambda}{B'_{k,\,z}}=\mathsf{ord}_{\lam}{B'_{k}},
 \]
 where $B_k(\bcdot)=B_{W_k}(\bcdot)$ is the minimal Blaschke product corresponding to $W_k$.
 Hence the numbers $q_z(\nu,j,k)$ does not depend on $z$. Also, since $\mathcal{M}_{\D}(z_1,\,z_2)
 =\mathcal{M}_{\D}(e^{i\theta}z_1,\,e^{i\theta}z_2)$, we have:
 \begin{align*}
&\sub{\mu\in\sigma\left(B_{k,\,z}(W_{L(k)})\right)}{\max}
\prod_{\nu\in\sigma\left(B_{k,\,z}(W_{G(k)})\right)}
{\mathcal{M}_{\D}\left(\intf{}{}{\mu}{\psi_k(\zt_{L(k)})},\,\intf{}{}{\nu}{\psi_k(\zt_{G(k)})}
\right)}^{q(\nu,\,G(k),\,k,\,z)}\\=&
\sub{\mu\in\sigma\left(B_k(W_{L(k)})\right)}{\max}
\prod_{\nu\in\sigma\left(B_k(W_{G(k)})\right)}
{\mathcal{M}_{\D}\left(\intf{}{}{\mu}{\psi_k(\zt_{L(k)})},\,\intf{}{}{\nu}{\psi_k(\zt_{G(k)})}
\right)}^{q(\nu,\,G(k),\,k)}.
\end{align*}
The above equality is also true when we replace $G(k)$ by $L(k)$ and vice-versa. It follows from this 
that when $\Omega=\D$, the statement of Theorem~\ref{T:3pt_nec} reduces to that of Theorem~1.4 in
\cite{chandel:3spintp20}.
\end{obs}

\begin{remark}\label{Rm:depenofmaint3pointZ}
It would be interesting to investigate for what domains $\Omega$, the statement of
Theorem~\ref{T:3pt_nec} does not depend on the choice of point $z\in\Omega$. As we have observed in
Observation~\ref{Obs:coromaintD} that this is the case when $\Omega=\D$. This question is, of course, 
related to the behaviour of Carath\'{e}odory extremal $G_{\Omega}(\lam,\,z;\,\bcdot)$ as $z$ varies over 
$\Omega$ and $\lam$ is any fixed point in $\Omega$.
\end{remark}

\section*{Acknowledgements} 
The author wishes to thank the anonymous referee for pointing out suggestions and problems in improving the presentation of the 
present version of the article. 
Most of the work done in this article was carried out while the author was an institute postdoctoral fellow
at Indian Institute of Technology Bombay. He is extremely grateful to the institute for all the support that he got during this period.
The author also wishes to thank Prof. Gautam Bharali, in particular, e-mail conversations with him helped the author to find an erroneous remark
in an earlier draft of this article and also discussions with him led to a better presentation of the first half of this article.

\end{document}